\documentclass{amsart} 
\usepackage{amsmath,amscd} 
\usepackage{latexsym,amsmath,amssymb,mathrsfs,setspace,enumerate}
\usepackage[T1]{fontenc}
\usepackage[utf8]{inputenc}
\usepackage{lmodern} 
\usepackage{amssymb}
\usepackage{enumitem}
\usepackage{tikz-cd}
\usepackage{setspace}
\usepackage{graphics}
\usepackage{graphicx}
\usepackage[utf8]{inputenc}
\usepackage[english]{babel}
\usepackage{fancyhdr}
\usepackage{tikz}
\usepackage{tikz-cd}
\usetikzlibrary{decorations.markings,positioning}
\theoremstyle{plain}
\usepackage[all]{xy}
\usepackage{latexsym,amsmath,amssymb,mathrsfs,setspace,enumerate,tikz}
\newtheorem{lemma}{Lemma} 
\newtheorem{prop}{Proposition} 
\newtheorem{thm}{Theorem}
\newtheorem{exam}{Example}

\newtheorem{rem}{Remark}
\newtheorem{exer}{Exercise}

\newtheorem{defn}{Definition} 
\newcommand{\clc}{{\mathcal C}}
\newcommand{\cld}{{\mathcal D}}

\usepackage{tikz-network}
\usetikzlibrary{automata,positioning}
\usepackage{tikz-cd}
\usetikzlibrary{arrows,chains,matrix,positioning,scopes}

\tikzcdset{%
	triple line/.code={\tikzset{%
			double distance = 3pt, 
			double=\pgfkeysvalueof{/tikz/commutative diagrams/background color}}},
	quadruple line/.code={\tikzset{%
			double distance = 5.3pt, 
			double=\pgfkeysvalueof{/tikz/commutative diagrams/background color}}},
	Rrightarrow/.code={\tikzcdset{triple line}\pgfsetarrows{tikzcd implies cap-tikzcd implies}},
	RRightarrow/.code={\tikzcdset{quadruple line}\pgfsetarrows{tikzcd implies cap-tikzcd implies}}
}

\title{Cross-connection semigroups amalgam of a vector bundle}
\author {P. G. Romeo }
\address{Dept. of Mathematics, Cochin University of Science and Technology, Kochi, Kerala, INDIA.}
\email{$romeo_-parackal@yahoo.com$}
\thanks{This paper is based on a talk deliveried at ICRA 21- 2024, Shanghai Jiao Tong University, Thanks to NBHM India, CSIR India and SJTU Shanghai} 
\date{}

\begin{document}
	
	\begin{abstract}
	Cross-connections of normal categories was introduced by K.S.S. Nambooripad while discussing the structure of regular semigroups and via this cross-connections he obtained a beautiful representetion of regualr semigroup called the cross-connection semigroup (see cf.\cite{kss}).	Subsequently cross-connection representation of  various other semigroups such as concordant semigroups, semigroup of endomorphisms of a vector space are also described (cf.\cite {pgr}\cite{paa}). In this paper we describe the semigroup amalgam of cross-connection semigroups of the fibers of a vector bundle.
	\end{abstract}
	\maketitle
	
	\section{Preliminaries}
	We begin by briefly recalling basic concepts  in semigroup theory, category theory and vector bundles.
	A semigroup is a structure $(S,\circ)$ where $S$ is a non empty set and $\circ$ is an associative binary operation on $S$, when the binary composition is evident then we simply denote the semigroup by $S$. Semigroups are more general structure then groups. Obviously every group is a semigroup.
	\begin{exam}
		All $n\times n$ matrices with entries from $\mathbb{R}$ or $\mathbb{C}$ with the usual 
		multiplication of matrices is a semigroup which is not a group.
	\end{exam} 
	\begin{exam}
		All transformations on a non empty set $X$  with composition of maps as binary operation is
		the full transformation semigroup written as  $T_X$.
	\end{exam}
	An element $a\in S$ is a regular element in $S$ if there is an $x\in S$ such that $axa=a$. If every element in $S$ are regular then $S$ is a regular semigroup. An element $e\in S$ with $ee=e$ is called an idempotent element.

\begin{defn}
	For $a\in S$, the set $Sa= \{ xa : x\in S \}$ is the principal left ideal generated by $a$. Inparticular when $a$ is an idempotent $e$, then $Se$ is the principal left ideal generated by idempotent $e$.
\end{defn}
 Similarly we have the principal right ideal $aS$ and two sied ideal $SaS$. 
When two elements in a semigroup generate the same principal left ideals then they are said to $\mathcal{L}$- related. This is an equivalence relation on the semigroup and is called the Green's $\mathcal{L}$- relation. When $a, b\in S$ generate the same principal right ideals then they are $\mathcal{R}$-related. 
The intersection of $\mathcal{L}$ and $\mathcal{R}$ is the Green's $\mathcal{H}$- relation and their join the $\mathcal{D}$-relation. These relations are very significant in the study of structure of semigroups and collectively they are known as Green's relations.

	In the following we briefly recall basic notions related to category theory with emphasis to subobject relation factorization and normal categories. 
	\begin{defn} A category $\mathcal{C}$ consists of the following data
		\begin{enumerate}
			\item a class $ \nu\mathcal{C} $ called the class of vertices or objects
			\item a class $\mathcal{C}$ of disjoint sets $\mathcal{C}(a,b)$, one for each pair $(a,b) \in \nu\mathcal{C} \times  \nu\mathcal{C}  $, an element $f \in \mathcal{C}(a,b)$ called a morphism (arrow) from $a$ to $b$, written $ f: a \rightarrow b; \, a = dom\,f$ and $b = cod \,f$
			\item For $a,b,c \in  \nu\mathcal{C}$, a map 
			$$ \circ : \mathcal{C}(a,b) \times \mathcal{C}(b,c) \rightarrow \mathcal{C}(a,c)$$
			$$ (f,g) \rightarrow f\circ g $$
			called the composition of morphisms in $\mathcal{C}$
			\item For each $ a \in \nu\mathcal{C}$, a unique $1_a \in \mathcal{C}(a,a)$ is called the identity morphism on $a$
		\end{enumerate}
		
		These must satisfy the following axioms:
		\begin{enumerate}
			\item[Cat 1] The compostion is associative: for $f \in \mathcal{C}(a,b)$, $g \in \mathcal{C}(b,c)$ and $h \in \mathcal{C}(c,d)$ we have
			$$ f\circ(g\circ h) = (f\circ g)\circ h.$$
			\item[Cat 2] For each $ a \in \nu\mathcal{C},\, f \in \mathcal{C}(a,b),\, g \in \mathcal{C}(c,a),\quad 
			1_a \circ f =f \,\text{and}\,  g \circ 1_a = g.$		
		\end{enumerate}
	\end{defn}

The following are some examples of categories

\begin{itemize}
	\item \textbf{Set}: the category in which objects are sets and morphisms are functions between sets.
	\item \textbf{Grp}: Category with groups as objects and homomorphisms as morphisms.
	\item \textbf{Groups}:  A group is a category with one object in which every arrow has a(two sided) inverse under composition.
\end{itemize}

A functor $F: \mathcal{C}\rightarrow \mathcal{D}$ from a category $\mathcal C$ to a category $\mathcal D$ consists of a vertex map $\nu F: \nu \mathcal{C} \rightarrow \nu \mathcal{D}$ which assigns to each $ a\in \nu \clc $ a vertex $ F(a) \in  \cld $ and a morphism map 
$F: \clc\rightarrow \cld$ which assigns to each morphism $f: a\rightarrow b,$ a morphism 
$ F(f): F(a)\rightarrow  F(b) \in \mathcal D$ such that 
\begin{enumerate}
	\item $ F(1_a) = 1_{F(a)}$ for all $a \in \nu \mathcal C;$ and
	\item$ F(f)F(g) = F(fg)$ for all morphisms $f, g \in \mathcal C$ for which the composition $fg$ exists.
\end{enumerate}

\begin{exam}
	The power set functor $\mathcal{P}: \textbf{Set}\rightarrow\textbf{Set}$. Its object function assigns to each object $X$ in $\textbf{Set}$ the usual power set $\mathcal{P}X$  with elements all subset  $ S\subset X$; its arrow function assigns to each $f:X\rightarrow Y$ that map $\mathcal{P}f:\mathcal{P}X\rightarrow\mathcal{P}Y $ which send each $S\subset X$ to its image $fS\subset Y $.
\end{exam}

Let $ \mathcal{C}$  and $ \mathcal{D} $ be two categories and  $F , G:\mathcal{C}\rightarrow \mathcal{D}$  be two functors. A \textbf{natural transformation} $ \eta:  F\rightarrow G$ is a family $ \{\eta_{a} : F(a) \rightarrow G(a) | a \in \nu {\mathcal C \}} $ of maps in $ \mathcal{D} $ such that for every map $ f: a \rightarrow b $ in $ \mathcal{C} $, yields a commutative diagarm in $ \mathcal{D} $\\
The map $ \eta_{a} $ are called the components of $ \eta $. If each component of $\eta$ is an isomorphism then $\eta$ is called a natural isomorphism .

\begin{defn}
	If $S: \mathcal{D} \rightarrow  \mathcal{C}$ be a functor and $c$ an object of $\mathcal{C}$, a universal arrow from $c \in \nu \clc $ to the functor $S$ is a pair $(r,u)$ where 
	$r \in \nu  \mathcal{D} $ and an arrow $u :c\to Sr$ of $\mathcal{C}$ such that given any pair $(d,f)$ with $d$ an object in $\mathcal{D}$ and $f:c\to Sd$ an arrow of $\mathcal{C}$, there is a  unique $f':r\to d$ of $\mathcal{D}$ with $Sf'\circ u=f$. 
\end{defn}

A preorder $\mathcal P$ is a category such that for any $ p, p'\in v\mathcal P$, the hom-set $\mathcal P(p, p')$ contains at most one morphism. 
In this case, the relation $\subseteq$ on the class $v\mathcal P$ of objects of 
$\mathcal P$ is defined by $p\subseteq p'$ if $\mathcal P(p,p') \neq \emptyset$
is a quasi- order and $\mathcal{P}$ is said to be a strict preorder if $\subseteq$ is a partial order.

\begin{defn} (Category with subobjects) Let $\mathcal C$ be a small category and $\mathcal P$ be a subcategory of $\mathcal C$ such that $\mathcal P$ is a strict preorder with $v\mathcal P= v\mathcal C.$ Then $(\mathcal C, \mathcal P)$ is a $category\; with\; sub\;objects$ if 
	\begin{enumerate}
		\item every $f\in \mathcal P$ is a monomorphism in $\mathcal C$
		\item if $f=hg$ for $f,g\in \mathcal P,$ then $h\in \mathcal P.$
	\end{enumerate}
\end{defn}
\begin{exam}
	In categories $Set, Grp, Vect_K, Mod_R $ the relation on objects induced  by the usual set inclusion is a subobject relation.
\end{exam}

In a category $(\mathcal C, \mathcal P)$ with subobjects, morphisms in $\mathcal P$ are  called inclusions. If $c'\to c$ is an inclusion, we write $c'\subseteq c$ and denotes this inclusion by $j^c_{c'}$. An inclusion $j^c_{c'}$ splits if there exists $q:c\to c'\in \mathcal C$ such that $j^c_{c'}q = 1_{c'}$ and the morphism $q$ is called a retraction.

\begin{defn}
	A morphism $f$ in a category $\mathcal{C}$ with subobjects is said to have \textit{factorization} if $f$ can be expressed as  $ f = pm$ where $p$ is an epimorphism and $m$ is an embedding.
\end{defn}

A normal factorization of a morphism $f\in \mathcal C(c,d)$ is a factorization of the form $f=quj$ where $q:c\to c'$ is a retraction, $u: c'\to d'$ is an isomorphism and $j= j^d_{d'}$ is an inclusion where $c', d'\in v\mathcal C$ with $c'\subseteq c,\; d'\subseteq d$. The morphism $qu$ is known as the epimorphic component of $f$ and is denoted by $f^{\circ}$.  

\begin{defn}
	Let $\mathcal C$ be a category with subobjects and $d\in v\mathcal C$.  A map $\gamma: v\mathcal C \to \mathcal C$ is called a  cone from the base $v\mathcal C$ to the vertex $d$ if 
	\begin{enumerate}
		\item $\gamma(c)\in \mathcal C(c,d)$ for all $c\in v\mathcal C$
		\item if $c\subseteq c'$ then $j^{c'}_c\gamma(c') = \gamma(c)$
	\end{enumerate}
\end{defn}

For a  cone $\gamma$ denote by $c_\gamma$ the vertex of $\gamma$ and for $c\in v\mathcal C$, the morphism $\gamma(c): c\to c_\gamma$ is called the component of $\gamma$ at $c$.
A cone $\gamma$ is said to be normal if there exists $c\in v\mathcal C$ such that $\gamma(c): c\to c_\gamma$ is an isomorphism. We denote by $T\mathcal C$ the set
of all normal cones in 	$\mathcal C$

\begin{defn}
	A category $\mathcal C$ with subobjects is called a normal category if the following holds
	\begin{enumerate}
		\item any morphism in $\mathcal C$ has a normal factorization
		\item every inclusion in $\mathcal C$ splits
		\item for each $c\in v\mathcal C$ there is a normal cone $\gamma $ with vertex $c$ and $\gamma(c) = 1_{c}.$
	\end{enumerate}
\end{defn}

Observe that given a normal cone $\gamma$ and an epimorphism $f: c_\gamma \to d$ the map 
$\gamma * f : a \to \gamma(a)f$ from $v\mathcal C$ to $\mathcal C$ is a normal cone with vertex $d$.

	\begin{rem}
		The set of all normal cones $T\mathcal C$ in $\mathcal C$ with the cone composition 
		$$\gamma^1\cdot \gamma^2= \gamma^1*(\gamma^2_{c_{\gamma^1}})^{\circ}$$
		is a regular semigroup.
	\end{rem}

\begin{exam} For the regular semigroup of normal cones  $T\mathcal C$ in normal category $\mathcal C$ the category of principal left and right ideals $L(T\mathcal C)$ and 
$R(T\mathcal C)$ normal categories.	
\end{exam}

	For each $\gamma\in T\mathcal C,\, H(\gamma, -):\mathcal C \to \bf Set$ gives a set valued functor. A  category whose objects are such set valued functors together with natural tranformations between such functors as morphisms is a functor category and is termed as the normal dual of $\mathcal C$, denoted by $N^*\mathcal C$. 
	
 That is, $N^*\mathcal C$ is the category whose object set is  
	$$vN^*\mathcal C= \{ H(\gamma,-) :\gamma \in T\mathcal C\}$$ and morphisms 
	$\sigma : H(\gamma,-) \to H(\gamma',-)$ given by   	
	\[
	\xymatrix{
		H(\gamma, -) \ar[r]^{\eta_{\gamma}} \ar[d]_{\sigma} & \mathcal{C}(c_{\gamma},-) \ar[d]^{\mathcal{C}(\hat{\sigma},-)}\\
		H(\gamma', -) \ar[r]_{\eta_{\gamma'}} &\mathcal{C}(c_{\gamma'},-) 
	}	
	\]
	where $\hat{\sigma}:c_{\gamma'}\to c_{\gamma}\in \mathcal{C}$.	

	\begin{thm}
		Let $\mathcal C$ be a normal category. Define $G$ on objects and morphisms of 
$R(T\mathcal C)$ as follows 
		$$vG(\gamma T\mathcal C)=H(\gamma,-)$$
		and for $\lambda :\gamma T\mathcal C\to \gamma' T\mathcal C$ the following diagam commutes
		\[
		\xymatrix {
			H(\gamma, -) \ar[r]^{\eta_{\gamma}} \ar[d]_{G(\lambda)} & \mathcal{C}(c_{\gamma},-) \ar[d]^{\mathcal{C}(\tilde{\gamma},-)}\\
			H(\gamma', -) \ar[r]_{\eta_{\gamma'}} & \mathcal{C}(c_{\gamma'},-) 
		}	
		\]
		
		Then $G:R(T\mathcal C)\to N^*\mathcal C$ is an isomorphism of normal categories.
	\end{thm}	

	Let $F:\mathcal C\to \mathcal D$ be an isomorphism of normal categories and $\gamma$ is a normal cone in $\mathcal C$, then 
	$$F(\gamma)(F(c))=F(\gamma(c))$$
	defines a normal cone in $\mathcal D$ and we denote this cone by $F(\gamma)$.
	\begin{lemma}
		Let $F:\mathcal C\to \mathcal D$ be an isomorphism of normal categories. For each 
		$H(\gamma, -)\in vN^*\mathcal C$ and $\eta_{\gamma}\mathcal C(f,-)\eta_{\gamma'}^{-1}\in  N^*\mathcal C$ define 
		\begin{align*}
			F_*(H(\gamma,-)) &= H(F(\gamma),-),\\
			F_*(\eta ) &=\eta_{F(\gamma)}\mathcal D(F(f),-)\eta_{F(\gamma')}^{-1}.
		\end{align*}
		Then $F_*$ is an isomorphism of $N^*\mathcal C$ onto $N^*\mathcal D$.
	\end{lemma}

	\begin{thm}
		Let $\mathcal C$ be a normal category. Then $\Gamma$ given by
		$$\Gamma=G\circ F_*$$
		then $\Gamma: R(T\mathcal C )\to N^*L(T\mathcal C)$ is an isomorphism and is termed as a connection.
	\end{thm}
\begin{rem}
Analogous to the above theorem there is local isomorphism $\Delta :L(T\mathcal C)\to N^* R(T\mathcal C )$ and is termed as the dual connection.
\end{rem}

Let $\mathcal C$ and $\mathcal D$ be normal categories,  
$\Gamma :\mathcal D \to N^*\mathcal C$ and  $\Delta :\mathcal C \to N^*\mathcal D$ are connection and dual connection. Then $\Gamma$ and $\Delta$ induces bi- functors $\Gamma(-,-): \mathcal D\times \mathcal C\to \bf{Set}$ and $ \Delta (-,-): \mathcal C\times \mathcal D\to \bf{Set}$ such that there is a natural bijection
$$\chi: \Gamma(-,-)\to \Delta(-,-).$$

\section{Cross-connections of normal categories}

	Given a connection $\Gamma :\mathcal D \to N^*\mathcal C$, then one can construct a connection from $\mathcal {C}_{\Gamma}$ with $N^*\mathcal D$, where $\mathcal {C}_{\Gamma}$ is an ideal of $\mathcal {C}$.
	\begin{thm} Let $\mathcal C$ and $\mathcal D$ be normal categories and $\Gamma :\mathcal D \to N^*\mathcal C$ be a connection. Then there esists a connection 
		$\Gamma^* :\mathcal {C}_{\Gamma} \to N^*\mathcal D$ such that, for $c\in \mathcal {C}_{\Gamma}$ and $d\in v\mathcal D, \, c\in M\Gamma(d)$ if and only if $d\in M\Gamma^*(c)$. 
	\end{thm}
	
	The connection $\Gamma^* :\mathcal {C}_{\Gamma} \to N^*\mathcal D$ is called the  dual connection of $\Gamma :\mathcal D \to N^*\mathcal C$. 
	\begin{defn}Let $\mathcal C$ and $\mathcal D$ be normal categories. A cross-connection is a  triple  $(\mathcal D, \mathcal C,\Gamma)$ where $\Gamma :\mathcal D \to N^*\mathcal C$ is a local isomorphism such that for every $c\in v\mathcal C$ there is some $d\in v\mathcal D$ such that $c\in M\Gamma(d)$.
	\end{defn}
		
	Let $\mathcal C$ and $\mathcal D$ be normal categories and $\Gamma :\mathcal D \to N^*\mathcal C$ a connection. Consider 
	$$E_{\Gamma}= \{(c,d): c\in vC_{\Gamma}, d\in v\mathcal C \,\text{and}\, c\in M\Gamma(d)\}$$
	and $E_{\Gamma^*}$ analogously. 
	Then for each $(c,d)\in E_{\Gamma}, \gamma(c,d)$ denote the unique cone in $\mathcal C$ such that 
	$$c_{\gamma(c,d)} =c\quad\text{and}\quad \Gamma(d)=H(\gamma(c,d),-)$$
	Also every cone on $\Gamma(c,d)$ is expressed as $\gamma(c',d)\ast f^{\circ}$ with $f\in \mathcal{C}(c',c)$ and every cone on $\Gamma^*(c,d)$ is $\gamma^*(c,d')\ast g^{\circ}$ with $g\in \mathcal{D}(d',d)$
	
	Given  
	$(\mathcal D, \mathcal C,  \Gamma)$ a cross-connection. Then for each $(c,d)\in \mathcal C\times \mathcal D$ there is a natural bijection  
	$$\chi_{\Gamma(c,d)}: \Gamma(c,d)\to \Gamma^*(c,d)$$ 
	and a pair of cones 
	$(\gamma, \gamma^*)$ with $\gamma\in \Gamma(c,d)$ is a linked pair if $\gamma^*= \chi_{\Gamma(c,d)}(\gamma)$.
	
	The collection of these linked pair of cones together with the binary composition defined by
	$$(\gamma,\gamma^*) \circ (\delta, \delta^*)=(\gamma\delta, \delta^* \gamma^*) $$
	is a semigroup and is called the corss-connection semigroup which is denoted by $\tilde{S}\Gamma$.

	\begin{exer}
		Fully describe the  cross-connection representation of a regular semigroup.\\
		Hint:\,step-1. Consider the category $L(S)$ with
		\begin{align}
			vL(S) &=\{Se: e\in E(S)\}\\
			Hom(Se,Sf)&= \{\rho(e,u,f): u\in eSf\}
		\end{align}
		where $\rho(e,u,f)$ are right translations. Similarly the category $R(S)$ and show that these are normal categories.\\
		step-2. Describe the local isomorohisms $\Gamma S:R(S)\to N^*L(S)$ and $\Gamma^* S:L(S)\to N^*R(S)$ and the bijection $\chi_S: \Gamma S(-,-)\to \Gamma^*S(-,-)$ whose components are given by
		$$\chi_S(Se,fS): \rho^f\ast \rho(f,u,e)^{\circ}\mapsto \lambda^e\ast \lambda(e,u,f)^{\circ}$$ 
	\end{exer}

\section{Cross-connections of singular endomorphisms $\it {Sing(V )}$ on $V$}

It is well known that the singular endomorphisms on any vector space $V$ is a  regular semigroup with the multiplication as composition of morphisms and is written as $\it {Sing(V )}$.
The following properties of ideal categories of $\it {Sing(V )}$ can easily be observed (see cf.\cite{paa}).

\begin{itemize}
	\item The category of principal left ideals of $\it {Sing(V )}$
	is characterized as the subspace category $\mathcal S(V)$ of proper subspaces(image spaces) of $V$ with linear transformations as morphisms and the principal right ideals of $\mathcal S(V)$ are also characterized by subspaces (null spaces ) of $V$.
	
	\item The normal dual of the subspace category is characterized in terms of the algebraic dualspace of $V$ and call it the annihilator category $\mathcal A(V)$. When $V$ is finite dimensional, it is shown that the annihilator category $\mathcal A(V)$ is isomorphic to the category 
	$\mathcal S(V^*)$ of proper subspaces of $V^*$ where $V^*$ is the algebraic dual space of $V$.	
\end{itemize}
	
A cone $\gamma \in \mathcal S(V)$ is certain mapping from $v\mathcal S(V) \to \mathcal S(V)$. A cone $\gamma $ is called normal if at least one component of $\gamma$ is an isomorphism. The category $\mathcal S(V)$ is a normal category and its set of all normal cones under the binary operation of cone composition forms the regular semigroup $\mathcal{TS}(V)$ (cf.\cite{paa}). 
Also, given an element $\alpha\in \it {Sing(V )}$, we have a normal cone $\rho^{\alpha}$ called the principal normal cone with vertex $Im\,{\alpha}$ and the map 
$$\alpha \mapsto \rho^{\alpha}$$
is an isomorphism of regular semigroups.  

\begin{prop}
	Let $V$ be a vector space. The semigroup of all normal cones $\mathcal{TS}(V)$ in the category $\mathcal S(V)$ of principal left ideals of $\it {Sing(V )}$ is isomorphic to $\it {Sing(V )}$.
\end{prop}

A cone $\gamma\in \mathcal{TS}(V)$ defines $H$-functors $H(\gamma,-): \mathcal S(V) \to Set$ and a  category whose objects are $H$-functors $H(\gamma,-):\mathcal S(V)\to Set$ where $\gamma$ is an idempotent normal cone in $\mathcal{TS}(V)$ with appropriate natural transformation as morpmisms. This functor category denoted by $ N^*\mathcal S(V)$ is the normal dual of the category $\mathcal S(V)$.

Let $A$ be any proper subspace of $V$. Define a new category viz., the annihilator category 
$\mathcal A(V)$ , whose objects are 
$$A^{\circ}=\{f\in V^*: f(v)=0 \,\text{for all}\,v\in A\}$$
and morphisms appropriative linear maps.  
It is easy to observe that $\mathcal A(V)$ is a normal category and is  isomorphic to the category $\mathcal S(V^*)$ of proper subspaces of $V^*$ where $V^*$ is the algebraic dual space of $V$  and being a normal category $\mathcal A(V)$ admits normal dual  $N^*\mathcal A(V)$.
 	
\begin{prop}(Proposition 5. cf.\cite{paa}) 
The normal dual $N^*\mathcal S(V)$ of the subspace category is isomorphic to $\mathcal A(V)$. Dually, $N^*\mathcal A(V)$ is isomorphic to $\mathcal S(V)$. The semigroup $\mathcal{TA}(V)$ of normal cones in $\mathcal A(V)$ is isomorphic to $Sing(V )^{op}$.
\end{prop}

Note that given normal categories  $\mathcal S(V)$ and $\mathcal A(V)$ there exists local isomorphisms 
\begin{gather*}
	\Gamma: \mathcal A(V)\to  N^*\mathcal S(V) \quad\text{and}\\
	\Delta : \mathcal S(V)\to N^*\mathcal A(V) \quad \text{such that}
\end{gather*}
$(\mathcal A(V),N^*\mathcal S(V),\Gamma)$ and $(\mathcal S(V),N^*\mathcal A(V), \Delta)$ are connections and dual connections repectively.\\

Hence, in the light of the above proposition we have 
$\Gamma: \mathcal A(V)\to \mathcal A(V)$ and $\Delta : \mathcal S(V)\to \mathcal S(V)$ with 
connections and dual connections $(\mathcal A(V),\mathcal A(V),\Gamma)$ and $(\mathcal S(V),\mathcal S(V), \Delta)$ respectively. 

\begin{thm} (Theorem 6. cf.\cite{paa}) Let $V$ be a finite diamensional vector space over $K$, then $N^* \mathcal S(V)$ is isomorphic as a normal category to $\mathcal S(V^*)$ 
\end{thm}

Now we proceed to describe all cross-connections between the nornal categories 
$\mathcal A(V),\,\mathcal S(V)$ and construct the semigroups arising from these categories.

\begin{defn}(Proposition 7.cf. \cite{paa})
	The $M$- set of the cone $\rho^e\in \mathcal S(V)$ is given by 
	$$M\rho^e= MH(\rho^e,-)= M((N_e)^{\circ})	= \{A\subseteq V : A \oplus N_e=V\}.$$
\end{defn}

\begin{defn}
A cross-connection from $\mathcal A(V)$ to $\mathcal S(V)$ is a triplet 
$(\mathcal A(V),\mathcal S(V), \Gamma)$ where $\Gamma: \mathcal A(V)\to \mathcal A(V)$ is a local isomorphism such that for every $A\in v\mathcal S(V)$, there
is some $Y\in  v\mathcal A(V)$ such that $A\in M(\Gamma (Y))$.
\end{defn}

 

Let $\epsilon$ be an automorphism of a vector space $V$, then the transpose 
$\epsilon^* :V^*\to V^*$ is also an automorphism. Hence define a functor 
$\Gamma_{\epsilon}$ on $\mathcal A(V)$ as follows; for a proper subspace $Y$ of $V^*$ and for 
$w^*: Y \to Z$  in $\mathcal A(V)$, let
$$\Gamma_{\epsilon}(Y)=\epsilon^* (Y)\quad\text{and}\quad \Gamma_{\epsilon}(w^*)=(\epsilon^{-1})^*w^*(\epsilon)^*=(\epsilon w\epsilon^{-1})^*$$

and in a similar way we have $(\Gamma^*)_{\epsilon}$ on  $\mathcal S(V)$ given by
$$(\Gamma_{\epsilon}^*)(A)=\epsilon(A)\quad\text{and}\quad (\Gamma_{\epsilon}^*)(f)=
\epsilon(f)=\epsilon^{-1}(f)\epsilon$$
thus we have the cross-connection $(\mathcal A(V),\mathcal S(V), \Gamma_{\epsilon})$. \\

Also it can be seen that given cross-connection 
$(\mathcal A(V),\mathcal S(V), \Gamma_{\epsilon})$ there are unique bifunctors 
$\Gamma_{\epsilon}(-,-):\mathcal A(V)\times \mathcal S(V)\to {\bf Set}$ given by, for all 
$(A,Y) \in \mathcal A(V)\times \mathcal S(V^*)$ and $(f,w^*): (A,Y)\to (B,Z)$ 
\begin{gather*}
	\Gamma_{\epsilon}(A,Y)= \{\alpha\in \mathcal{T}V : V\alpha\subseteq A\,\text{and}\, (A\alpha)^{\circ}\subseteq \Gamma_{\epsilon}(Y)\}\\
	\Gamma_{\epsilon}(f,w^*) :\alpha\mapsto (y\alpha)f=y((\alpha f)\quad\text{where $y$ is given by}\,y^*=\Gamma_{\epsilon}(w^*) 
\end{gather*}
and bifunctor $\Gamma_{\epsilon}^*(-,-): \mathcal S(V) \times \mathcal A(V)\to {\bf Set}$. \\

Moreover these bifunctors $\Gamma_{\epsilon}(-,-)$ and $\Gamma_{\epsilon}^*(-,-)$ are connected by a natural bijection  $\chi_{\Gamma_{\epsilon}}$  given by 
$$\chi_{\Gamma_{\epsilon}}(A,Y):\alpha\mapsto {\epsilon^{-1}}\alpha \epsilon $$
and the natural bijection $\chi_{\Gamma_{\epsilon}}$ provides a pair of cones 
$(\alpha, \, \beta)$ where  
$$\beta=\chi_{\Gamma_{\epsilon}}(\alpha)={\epsilon^{-1}}\alpha \epsilon .$$

The cross-connection semigroup associted with this cross-connections is given by 
$$\tilde{S}\Gamma_{\epsilon}= \{(\alpha, \epsilon^{-1}\alpha\epsilon) \quad \text{such that}\, \alpha\in 
\it {Sing(V )}  \}.$$
Further we have the following theorem 
\begin{thm} (cf. \cite{paa} Theorem 8)
	Every cross-connection semigroup arising from the cross-connections between the categories 
	$\mathcal A(V)$ and $\mathcal S(V)$ is isomorphic to $\it {Sing(V )}$. 
\end{thm}

\section{Bundles and vector bundles}
	A ${\it  fiber \, bundle\,}$ ( simply  a bundle) is a triple $\eta=(E,p,B)$ is where $p:E\to B$ is a map. The space $B$ is called the base space and the space $E$ is called total space and the map $p$ is called the projection of the bundle. For each $b\in B$, the space $p^{-1}(b)$ is called the ${\it  fibre}$ of the bundle over $b\in B$.
Intuitively a bundle is regarded as a union of fibres $p^{-1}(b)$ for $b\in B$ parametrized by $B$ and 'glued together' by the topology of the space $E$.

\begin{exam}
	A covering space is a(continuous, surjective) map p : X $\rightarrow$ Y such that for every y $\in$ Y there exist an open neighborhood U containing y such that $p^{-1}$(U) is homeomorphic to a disjoint union of open sets in X, each being mapped homeomorphically onto U by p .  Hence it is a  fibre bundle such that the bundle projection is a local homeomorphism.
\end{exam}	

\begin{exam}
	Given any space $B$, a product bundle over $B$ with fibre $F$ is the bundle $(B\times F ,p,B)$ where $p$ is the projection on the first factor.
\end{exam}	

\begin{defn} A bundle $\eta '=(E',p',B')$	is a subbundle of $\eta=(E,p,B)$ provided $E'$ is a subspace of $E$, $B'$ is a subspace of $B$ and $p'=p\vert \, E'\,:\,E'\to B'$.	
\end{defn}

\begin{exam}
	The tangent bundle over $S^n$ denoted as $(T,p,S^n)$ and the normal bundle $(N,q, S^n)$ are two subbundles of the product bundle $(S^n\times R^{n+1},p, S^n)$ whose total space are defined by the relation $(b,x)\in T$ if and only if the inner product $\langle b,x\rangle=0$ and by $(b,x)\in N$ if and only if $x=kb$ for some $k\in R$.
	
\end{exam}

Two bundles $\eta$ and $\zeta $ over $B$ are locally isomorphic if for each $b\in B$  there exists an open neighbourhood $U$ of $b$ such that $\eta\vert U$ and $\zeta\vert U $ are $U$-isomorphic.

\begin{defn} Consider the bundle  $(E,p,B)$. 	A map $s: B\to E$ is called a section (also cross section) of the bundle $(E,p,B)$  if $p\circ s =id_B$ 
\end{defn}

A space $F$ is a fibre of a bundle  $(E,p,B)$ provided every fibre $p^{-1}(b)$ for $b\in B$ is homoeomorphic to $F$.  A bundle  $(E,p,B)$ is trivial with fibre $F$ provided  $(E,p,B)$ is $B$-isomorphic to the product bundle $(E\times F,p,B)$.

	\begin{defn} A bundle $\eta $ over $B$ is locally trivial with fibre $F$ if $\eta $ is locally isomorphic with the product bundle $(B\times F,p,B)$ 
\end{defn}

A fibre bundle with fibre $F$ is also written as $(E,p,B,F)$, where $E,\,B$ and $F$ are topological spaces, and $p:E\to B$ is a continous surjection satisfying local trivality condition.  

A bundle morphism is a pair of  fibre preserving map between two bundles $\eta=(E,p,B)$ and $\eta'=(E',p',B')$ given by 
$$(u,f):(E,p,B)\to (E',p',B')$$
where $u:E\to E'$ and $f:B\to B'$ such that $p'u=fp$, ie., the diagram of arrows commutes.
In particular when $\eta=(E,p,B)$ and $\eta'=(E',p',B)$ bundles over $B$, then the bundle morphism 
$u:(E,p,B)\to (E',p',B)$ is a map $u: E\to E'$ such that $p=p'u$.

\begin{exam} Let $(E',p',B')$	is a subbundle of $(E,p,B)$ . If $f:B'\to B$ and $u:E'\to E$ are inclusion maps, then $(u,f): (E',p',B')\to (E,p,B)$ is a bundle morphism.
\end{exam}

\begin{defn} 
	The category of bundles, denoted ${\bf BUN}$, has as its objects all bundles $(E,p,B)$ and as morphisms $(E,p,B)$ to $(E',p',B')$ the set of all bundle morphisms.
\end{defn}
	For each space $B$, the sub category of bundles over $B$, denoted as ${\bf BUN_B}$ has objects bundles with base $B$ and $B$-morphisms as its morphisms.

\subsection{Vector bundles} A  bundle with an additional vector space structure on each fibre is called a vector bundle.
\begin{defn}
	A k- dimensional vector bundle over a field $F$ is a bundle $(E,p,B)$ together with the structure of a k- dimensional vector space on each fibre $ p^{-1}(b)$ such that the following local triviality condition is satisfied. Each point of B has an open neighborhood U and a U- isomorphism h: U $\times$ $F^{k}$ $\rightarrow$ $p^{-1}(U)$ such that the restriction b $\times$ $F^{k}$ $\rightarrow$ $p^{-1}(b)$ is a vectorspace isomorphism for each $b \in U$. 
\end{defn}

\begin{exam}
	If $M$ is a smooth manifold and $k$ is a nonnegative integer then 
	$$p:M\times \mathbb{R}^k\rightarrow M$$
	is a real vector bundle of rank $k$ over $M$ and is a trivial vector bundle.
\end{exam}

  \begin{exam}
	Let $M$ be a smooth $n-$ manifold, the collection of 
	tangent vectors $TM$ of $M$ is a vector bundle of rank $n$ over $M$. As a set it is given by disjoint union of tangent spaces of $M$.
	Elements of $TM$ can be thought of as a pair $(x,v)$ where $x$ is a point in $M$ and $v$ is a tangent vector to $M$ at $x$. There is a natural projection $\pi : TM \rightarrow M$ denoted $\pi(x,v) = x$ and the vector bundle is $(TM,\pi, M)$ where each fibre is the tangent space at some point of $M$. \\
	The cotangent bundle may describe in a similar way and is dented by $T^*M$.
 \end{exam}

\begin{exam}
	The tangent bundle over $S^n$, denoted $(T, p, S^n)$, and the normal bundle over $S^n$, denoted $(N,q, S^n)$ are two subbundles of the product bundle $(S^n\times R^{n+1}, p, S^n)$ whose total spaces are defined by the relation $(b,x)\in T$ if and only if the inner product $\langle b,x\rangle=0$ and by $(b,x)\in N$ if and only if $x=kb$ for some $k\in R$.
\end{exam}
\begin{defn}
	Let $p:V\to M$ is a vector bundle. A section of $p$ or $V$ is a smooth map $s:M\to V$ such that $p\circ s=id_M$ ie., $s(x)\in V_x$ for all $x\in M$.
\end{defn}
Note that if $s$ is a section then $s(M)$ is an embedded submanifold of $V$.

\begin{defn}
	Suppose $p:V\to M$ and $q:V'\to N$ are two real vector bundles. 
	\begin{itemize}
		\item{(i)} A smooth map $\tilde{f}:V\to V'$ is a vector bundle morphism if and only if $\tilde{f}$ decends to a map $f:M\to N$ such that $q\tilde{f}=fp$, ie., the diagram below commutes 
		\end{itemize}
	
	\begin{center}
		\begin{tikzpicture}
		\fill (2.4,2.4)  node[right] {$V'$};
		\fill (0.3,2.4)  node[left] {$V$};
		\fill (2.4,0.3)  node[right] {$N$};
				
		\fill (0.3,0.3) node[left] {$M$};
		\draw[->, very thick] (2.6,2.2) -- (2.6,.5)node[midway, right]{$q$};
		\draw[->, very thick] (0,2.2) -- (0,.5)node[midway, left]{$p$};
		\draw[->,  very thick] (.3,0.4) -- (2.3,0.4)node[midway, above]{$f$};
		\draw[->, very thick] (.3,2.5) -- (2.3,2.5)node[midway, above]{$\tilde{f}$};
		\end{tikzpicture}
	\end{center}
	and the restriction $\tilde{f}:V_x\to V_{f(x)}$ is linear for all $x\in M$.
		\begin{itemize}
			\item{(ii)} If $p:V\to M$ and $q:V'\to M $ are vector bundles, a vector bundle morphism $\tilde{f}:V\to V'$ is an isomorphism of vector bundles if $q \tilde{f}=p$. If such an isomorphism exists then $V$ and $V'$ are said to be isomorphic.
		\end{itemize}
	\end{defn}
	
		\begin{rem} The real line bundle $V\to S^1$ given by the infinite Mobious band is not 	  
		isomorphic to the trivial line bundle $S^1\times \mathbb{R} \to S^1$.
	\end{rem}
		
	\begin{rem}Note that all natural operations on vector spaces  such as taking qoutient vector spaces, dual vector spaces, direct sum of vector spaces, tensor product of vector spaces and exterior powers are all carry over to vector bundles.
	\end{rem}

\section{Category of fibers of a vector bundle and semigroup amalgam}
Now we consider the category whose objects are fibers of a vector bunles with morphisms as fiber maps. Since each fiber is a vector space $V_{\alpha}$, we have an associated endomorphism semigroup $\it {Sing(V_{\alpha} )}$ and cross-connection semigroup $\tilde{S_{\alpha}}\Gamma_{\epsilon}$. In the following we construct the semigroup amalgam of these cross-connection semigroups and thus obtained a semigroup amalgam associated to a vector bundle.

Consider the vector bundle $p:V \to M$ with base $M$ of dimension $k$. Then each fiber $p^{-1}(m)$  of the bundle $p:V \to M$ is a $k$ dimensional vector space $V_m$. Now it is easy to observe that any subspace of $V_m$ is a fibre of some vector bundle of dimension $"\leq"\, k$.  

\begin{defn}
For a vector bundle $p:V \to M$ of dimension $k$, there is a category termed as the category of fibers of vector bundles whose object is the set of fibers of diamension $"\leq"\, k$  and smooth linear maps between these objects as morphisms and we denote this category by $\clc_{fb}(M)$.
\end{defn}

The objects set in the category $\clc_{fb}(M)$ is given by 
$$v\clc_{fb}(M)= \{V_{\alpha} :V_{i}\,\text{vector space of }\,dim\,\leq k,\, i\in I,\,\text{where $I$-indexing set}\}$$ 
The endomorphism semigroup of $V_{i}$ is the semigroup $\it {Sing(V_{i} )}$ and its cross-connection semigroup is 
$$\tilde{S_{i}}\Gamma_{\epsilon}= \{(\alpha, \epsilon^{-1}\alpha\epsilon) \quad \text{such that}\, \alpha\in 
\it {Sing(V_i )}  \}.$$.

  Next we proceed to describe the semigroup amalgam of the semigroups 
 $\tilde{S_{i}}\Gamma_{\epsilon}$ and thus obtain a structure associated with a vector bundle.

  \subsection{Semigroup amalgam}
   A semigroup amalgam ${\mathcal U}= [\{S_i \,:\, i\in I\}; U; \{\phi_i :i\in I\}]$ consisting of a semigroup $U$, called thge core of the amalgam, a family $\{S_i \,:\, i\in I\}$ of semigroups disjoint from each other and from $U$ and a family of monomorphisms $\{\phi_i :U\to S_i, :i\in I\}$. Often we simplify the notation to ${\mathcal U}= [\{S_i\}, U]$ when the context allows.
  
  \begin{exam}
  Let $U= \{u, v, w, z\}$ be a null semigroup in which all products
  are equal to $z$. Let $S_1=U\cup \{a\} $ where $a\notin S_1, au=ua=v$, and all the other
  products in $S_1$ are set equal to $z$. Let $S_2=U\cup \{b\} $ where $b\notin S_1, 
  bv = vb = w,$ and all other products in $s_2$ are set equal to $z$. Then a straightforward   verification shows that both $S-1$ and $S_2$ are semigroups. Then 
  $[U; S_2;S_2;i_1;i_2]$ (where $i_1$ and $i_2$ are the embedding morphisms) is a semigroup amalgam.
\end{exam}
  
  Consider vector bundle $p:V\to M$ of dimension $k$ and the category of fibers of vector bundle $\clc_{fb}(M)$. For each object $V_i\in \clc_{fb}(M)$ these is an associated cross-connection semigroup $\tilde{S_i}\Gamma_{\epsilon}$ and hence we have  an indexed family of disjoint semigroups $\{\tilde{S_i}\Gamma_{\epsilon}: i\in I\}$
  
  \begin{thm} Let $p: V\to M$ be a vector bundle whose fibers is an indexed family of vectorspaces $\{V_i: i\in I \}$. For each isomorphism $\epsilon\in V_i$, there is a  cross-connection semigroup 
  	$$\tilde{S_i}\Gamma_{\epsilon}= 
  	\{(\alpha, \epsilon^{-1}\alpha\epsilon) \quad \text{such that}\,\, \alpha\in \it {Sing(V_{\alpha} )} \}.$$
  Then there is semigroup amalgam $${\mathcal U}= [\{\tilde{S_i}\Gamma_{\epsilon} \,:\, i\in I\}; U; \{\phi_i :i\in I\}]$$
  of the vector bundle  $p: V\to M$ where $U$ is the core of the amalgam and a family of injective homorphisms $\{\phi_i :U\to S_i, :i\in I\}$. 
  \end{thm}

  \begin{proof}
  We have $\{\tilde{S_i}\Gamma_{\epsilon} :i\in I\}$, a disjoint family of cross-connection
  semigroups of fibers $\{V_i:i\in i\}$.  
  Let ${\mathcal B_i}$ be a basis for ${V_i}$ and so each 
${\mathcal B_i}$'s are disjoint linearly independant sets.
  Now choose linearly independant set $W$ disjoint from each ${\mathcal B_i}$'s and $\mid W\mid \leq \mid {\mathcal B_i}\mid,\, \forall i $ and let $\tilde{W}$ be the vector space spanned by $W$. Then there is a cross-connection semigroup of $\tilde{W}$, given by $\tilde{W}\Gamma_{\epsilon}$, denote it by $U$. Now it is easy to see that there exists injective homorphisms $\{\phi_i :U\to S_i, :i\in I\}$. Hence we have the semigroup amalgam ${\mathcal U}= [\{\tilde{S_i}\Gamma_{\epsilon}\} ; U; \{\phi_i \}]$.
  \end{proof}

\end{document}